\documentclass[12pt,reqno]{amsart}
\usepackage{amsmath,amsthm,amsfonts,amssymb,times}
\usepackage[mathscr]{eucal}
\usepackage{verbatim}
\usepackage{url}
\setbox0=\hbox{$+$}
\newdimen\plusheight
\plusheight=\ht0
\def\+{\;\lower\plusheight\hbox{$+$}\;}

\setbox0=\hbox{$-$}
\newdimen\minusheight
\minusheight=\ht0
\def\-{\;\lower\minusheight\hbox{$-$}\;}

\setbox0=\hbox{$\cdots$}
\newdimen\cdotsheight
\cdotsheight=\plusheight
\def\cds{\lower\cdotsheight\hbox{$\cdots$}}

\makeatletter
\def\leqalignno#1{\displ@y \tabskip\z@ plus\@ne fil
  \halign to\displaywidth{\hfil$\@lign\displaystyle{##}$\tabskip\z@skip
    &$\@lign\displaystyle{{}##}$\hfil\tabskip\z@ plus\@ne fil
    &\kern-\displaywidth\rlap{$\@lign\hbox{\rm##}$}\tabskip\displaywidth\crcr
    #1\crcr}}
\makeatother

\newcommand{\eb}{\begin{equation}}
\newcommand{\ee}{\end{equation}}
\renewcommand{\Im}{\operatorname{Im}}
\newcommand{\df}{\dfrac}
\newcommand{\tf}{\tfrac}

\renewcommand{\Re}{\operatorname{Re}}
\renewcommand{\Im}{\operatorname{Im}}

\newcommand{\s}{{\sigma}}

 \renewcommand{\a}{\alpha}
\renewcommand{\b}{\beta}
\newcommand{\e}{\epsilon}

\renewcommand{\l}{\lambda}

\renewcommand{\Re}{\textup{Re}}
\renewcommand{\Im}{\textup{Im}}

\renewcommand{\(}{\left\(}
\renewcommand{\)}{\right\)}
\renewcommand{\[}{\left\[}
\renewcommand{\]}{\right\]}

\numberwithin{equation}{section}
 \theoremstyle{plain}
\newtheorem{theorem}{Theorem}[section]
\newtheorem{lemma}[theorem]{Lemma}
\newtheorem{corollary}[theorem]{Corollary}

\newtheorem{entry}[theorem]{Entry}

\allowdisplaybreaks

\begin{document}
\title[Ramanujan and Koshliakov Meet Abel and Plana]
{Ramanujan and Koshliakov Meet Abel and Plana}
\author{Bruce C.~Berndt, Atul Dixit, Rajat Gupta, Alexandru Zaharescu}
\address{Department of Mathematics, University of Illinois, 1409 West Green
Street, Urbana, IL 61801, USA} \email{berndt@illinois.edu}
\address{Department of Mathematics, Indian Institute of Technology Gandhinagar, Palaj, Gandhinagar 382355, Gujarat, India}\email{adixit@iitgn.ac.in}
\address{Department of Mathematics, Indian Institute of Technology Gandhinagar, Palaj, Gandhinagar 382355, Gujarat, India}\email{rajat\_gupta@iitgn.ac.in}
\address{Department of Mathematics, University of Illinois, 1409 West Green
Street, Urbana, IL 61801, USA; Institute of Mathematics of the Romanian
Academy, P.O.~Box 1-764, Bucharest RO-70700, Romania}
\email{zaharesc@illinois.edu}

\begin{abstract}
The neglected Russian mathematician, N.~S.~Koshliakov, derived beautiful generalizations of the classical Abel--Plana summation formula through a setting arising from a boundary value problem in heat conduction. When we let the parameter $p$ in this setting tend to infinity, his formulas  reduce to the classical Abel--Plana summation formula. Rigorous formulations and proofs of these summation formulas are given. In his notebooks, Ramanujan derived different analogues of the Abel--Plana summation formula.  One particular example provides a vast new generalization of the classical transformation formula for Eisenstein series, which we generalize in Koshliakov's setting.
\end{abstract}
\maketitle

\begin{center}\textbf{In Memory of M.~V.~Subbarao}\end{center}

\section{Introduction} For the past three centuries, general summation formulas have been very useful, especially for those working in number theory, combinatorics, analysis, and applied mathematics.  Among these summation formulas are the Euler--Maclaurin, Poisson, Vorono\"{\i}, and Abel--Plana summation formulas. As the title suggests, the Abel--Plana summation formula is the focus of the present paper. It has many applications. For example, it readily gives Hermite's formula for the Hurwitz zeta function $\zeta(s, a)$ \cite[p.~609, Formula \textbf{25.11.27}]{nist}. It also provides a useful representation for the discrete Laplace transform of a function of semi-exponential type \cite[p.~329--332]{henrici2}.

There exist various extensions and generalizations of the Abel--Plana summation formula; see, for example, \cite{bellucci}, \cite{saaryan1} and \cite{saaryan2}. See also \cite{butzer} for an informative article on the connections among Euler--Maclaurin, Abel--Plana, Poisson summation formulas and the approximate sampling formula.

The work of the Russian mathematician, N.~S.~Koshliakov, is not well-known, perhaps because most of his papers were written in Russian, although he also published in English and German.  In particular, he derived  beautiful generalizations and variants of the Abel--Plana summation formula.  From their appearances in several of his papers, it is clear that he dearly loved the Abel--Plana summation formula and his analogues, all of which have largely been ignored \cite{koshliakov},  \cite{koshliakov7}, \cite[Equation 4]{koshliakovsomesum1}, \cite[Equation 7]{koshliakovsomesum2}, \cite[p.~46, Equation (I)]{koshliakovmellinappl}, \cite{koshliakov3}.  An atypical feature of Koshlikov's summation formulas and his variants is that generally the sums are not over the positive integers or the non-negative integers, but instead over a sequence of positive numbers $\lambda_n$, which are solutions of a certain transcendental equation.  They lie in intervals of the form $(n-\tf12,n)$, and when $\lambda_n\to n, n\in\mathbb{N}$, Koshliakov's summation formula reduces to the Abel--Plana summation formula.  Koshliakov's summation formula has an uncommon origin in that it arises from an eigenvalue problem in physics, which, as given in \cite{koshliakovmatsb}, is now briefly  explained.

Take a sphere of radius $r=R_2$ with temperature $v(t)$. Assume that $v(0)=0. $  Let a heat source be placed on a spherical surface of radius $R_1$,\hspace{1mm} where $0<R_1<R_2$.  Let the rate per unit time of the propagation of heat from the entire surface be $Q(t)$. The problem considered by Koshliakov is concerned with finding the temperature of the sphere $r=R_2$ at time $t>0$. Let $k$ be the thermal conductivity, $h$ the emissivity of the surface, $c$ the specific heat, and $\rho$ the density of the material which forms the sphere. The heat equation relevant to the problem is
\begin{align}\label{he}
\frac{\partial v}{\partial t}&=a^2\frac{\partial^2 v}{\partial r^2}, \qquad a=\sqrt{c/(k\rho)};\nonumber\\
v|_{r=0}&=0,\qquad  \left.\frac{\partial v}{\partial r}\right|_{r=R_2}+\left.\left(H-\frac{1}{R_2}\right)v\right|_{r=R_2}=0, \qquad H=h/k.
\end{align}
Problems such as these are  important in the analytical theory of heat distribution.

Koshliakov's summation formula first appeared in \cite{koshliakov}, where there is no statement of the theorem.  The theorem's hypotheses are woven into the proof, for which details are few.  Therefore, our first task is to resurrect Koshliakov's formula, record it as a theorem with hypotheses, and give a rigorous proof with all necessary details in Theorem \ref{theorem1}. Koshliakov \cite[p.~53, Equation (1)]{koshliakov3} obtained an equivalent form of his formula (see Theorem \ref{theorem4} below) and also gave its proof. However, it involves results associated with a generalized zeta function $\zeta_p(s)$ introduced by Koshliakov whereas the proof in \cite{koshliakov} employs only contour integration and Cauchy's residue theorem. The function $\zeta_p(s)$ is defined by \cite[p.~6]{koshliakov3}
\begin{align}\label{zps}
\zeta_p(s):=\sum_{j=1}^{\infty}\frac{p^2+\lambda^2_{j}}{p\left(p+\frac{1}{\pi}\right)+\lambda^2_j}\cdot \frac{1}{\lambda^s_j},\qquad
\Re(s)>1
\end{align}
where $\lambda_j$ runs over the roots of the transcendental equation
\begin{align}\label{ce2}
p \sin(\pi \lambda)+\lambda \cos(\pi \lambda)=0.
\end{align}
Actually, this is one of the two generalized zeta functions whose theories Koshliakov develops in his manuscript \cite{koshliakov3}. The other generalized zeta function, namely, $\eta_p(s)$, is defined by \cite[p.~6]{koshliakov3}
\begin{equation}\label{eps}
\eta_{p}(s):= \sum_{k=1}^{\infty}\frac{(s,2 \pi p k)_k}{k^s},\hspace{5mm}\textup{Re}(s)>1,
\end{equation}
where
\begin{equation}\label{s}
(s,\nu k)_k:=\frac{1}{\Gamma(s)}\int_{0}^{\infty}e^{-x}\left(\frac{k\nu-x}{k\nu+x}\right)^kx^{s-1}\, dx.
\end{equation}
These two generalized zeta functions are designated as \emph{Koshliakov zeta functions} in \cite{kzf}.

The aforementioned summation formula of Koshliakov is one of several variations and generalizations of the Abel--Plana summation formula.  In Theorem \ref{theorem3}, we use Theorem \ref{theorem1} to prove perhaps a closer analogue of the Abel--Plana summation formula. Koshliakov \cite[Chapter 3, Equation (I)]{koshliakov3} gave another generalization of the Abel--Plana summation formula, whose proof again involves the theory of his generalized zeta functions. It is stated in Theorem \ref{koshalt} below.

Ramanujan also discovered the Abel--Plana summation formula and recorded it twice, once in Chapter 13 of his second notebook  and later on page 335 in the unorganized pages of his second notebook \cite{nb}, \cite[pp.~220--221]{II}, \cite[p.~411]{V}.  On page 335, he also recorded three analogues of the Abel--Plana summation formula.  Another goal of the present paper is to demonstrate that all three of Ramanujan's general theorems can be derived from Koshliakov's versions of the Abel--Plana summation formula.  We do not know Ramanujan's proofs of these three summation formulas.  The proofs given by the first author in \cite[pp.~413--416]{V} employ contour integration, which almost certainly is not how Ramanujan proved these three formulas.

\section{The Abel--Plana Summation Formula}
    The rendition of the Abel--Plana summation formula that we give below is taken from P.~Henrici's text \cite[p.~274]{henrici}.

\begin{theorem}\label{abelplana}
Let $\varphi(z)$ be analytic for $\Re\,z\geq0$, and suppose that either
$$\sum_{n=0}^{\infty}\varphi(n) \qquad \text{or} \qquad \int_0^{\infty}\varphi(x)dx$$
converges.  Assume further that
$$\lim_{y\to\infty}|\varphi(x\pm iy)|e^{-2\pi y}=0,$$
uniformly in $x$ on every finite interval, and that
$$\int_0^{\infty}|\varphi(x\pm iy)|e^{-2\pi y}dy$$
exists for every $x\geq0$ and tends to $0$ as $x\to\infty$.  Then
\begin{equation}\label{abel}
\sum_{n=0}^{\infty}\varphi(n)=\frac12\varphi(0)+\int_0^{\infty}\varphi(x)dx
+i\int_0^{\infty}\df{\varphi(iy)-\varphi(-iy)}{e^{2\pi y}-1}dy.
\end{equation}
\end{theorem}

\section{Koshliakov's Generalization of the Abel--Plana Summation Formula}

In order to state Koshliakov's generalization of the Abel--Plana summation formula, it is first necessary to make some definitions and prove two lemmas.

Define
\begin{equation}\label{sig}
\sigma(t):=\df{p+t}{p-t},
\end{equation}
where $p>0$ is real. Note that
\begin{equation}\label{sigma}
\sigma(-t)=1/\sigma(t).
\end{equation}
It is to be noted that the parameter $p=R_2H-1$ in the setting of \eqref{he} as mentioned by Koshliakov in \cite{koshliakovmatsb}.

\begin{lemma} The zeros of $\sigma(iz)=e^{-2\pi iz}$ are real, and if $z=x+iy$, they are roots of the equation
\begin{equation}\label{diophantine}
x\cos(\pi x)+p\sin(\pi x)=0.
\end{equation}
\end{lemma}

Throughout the sequel, $\mathbb{H}=\{z: \Im \, z>0\}$ and
$\mathbb{H^-}=\{z: \Im \, z<0\}$.

\begin{proof}Suppose that
\begin{equation}\label{roots}
\df{p+iz}{p-iz}e^{2\pi iz}=1.
\end{equation}
If $z\in\mathbb{H}$,
$$ \left|\df{p+iz}{p-iz}\right|  <1,\qquad    \left|e^{2\pi iz}\right|<1.$$
Thus, there are no roots of \eqref{roots} in $\mathbb{H}$.  If $z\in\mathbb{H}^{-}$, then
$$ \left|\df{p+iz}{p-iz}\right|  >1,\qquad    \left|e^{2\pi iz}\right|>1.$$
  In conclusion, all of the roots of \eqref{roots} are real.

If we replace $z$ by real $x\neq0$ in \eqref{roots} and equate real and imaginary parts of $(p+ix)e^{2\pi ix}=p-ix$, we find, respectively, that
\begin{align}
p\cos(2\pi x)-x\sin(2\pi x)-p&=0,\label{first}\\
x\cos(2\pi x)+p\sin(2\pi x)+x&=0.\label{second}
\end{align}
Rewrite \eqref{first} and \eqref{second} in their respective forms,
\begin{align}
-p\sin^2(\pi x)-x\sin(\pi x)\cos(\pi x)&=0,\label{firsta}\\
x\cos^2(\pi x)+p\sin(\pi x)\cos(\pi x)&=0.\label{seconda}
\end{align}
If $\cos(\pi x)=0$, then from \eqref{firsta}, $p=0$, which contradicts the fact that $p\neq0$.  Thus, $\cos(\pi x)\neq0$. If $\sin(\pi x)=0$, then from \eqref{seconda}, $\cos(\pi x)=0$.  However, we have already seen that $\cos(\pi x)\neq0$, and so we have another contradiction.  Hence, from \eqref{seconda},
\begin{equation}\label{diophantine5}
x\cos(\pi x)+p\sin(\pi x)=0.
\end{equation}
From \eqref{firsta} also, we can deduce \eqref{diophantine5}.
\end{proof}

\begin{lemma}\label{rooty} The non-negative values of $z$ such that
\begin{equation*}
\df{p+iz}{p-iz}e^{2\pi iz}=1
\end{equation*}
are $0,\lambda_1, \lambda_2, \dots \lambda_k,\dots$, where  $k-\tf12<\lambda_k<k$,
and satisfy the equation
\begin{equation}\label{rootsa}
\tan(\pi x)=-\df{x}{p}.
\end{equation}
\end{lemma}

\begin{proof}  Since $x\neq0$ and $\cos(\pi x)\neq0$, it is readily seen that \eqref{diophantine} and \eqref{rootsa} are equivalent.
 Now $\tan(\pi x)$ has a positive zero at $x=k$, where $k$ is a positive integer. We see that immediately to the left of $k$ there will be a positive root of \eqref{rootsa}, because $\tan(\pi z)$ becomes negative. This root will be greater than $k-\tf12$, because $\tan(\pi z)$ changes sign at $z=k-\tf12$. If we denote by $\lambda_k$ that root in $(k-\tf12,k)$, we complete the proof.
 \end{proof}

\begin{theorem}\label{theorem1} Let $f(z)$ denote an analytic function for $\Re(z)\geq 0$.  Set $z=x+iy$.  Suppose that
\begin{equation}\label{hypoth}
\int_0^{\infty}|f(x+iy)|dx <\infty,
\end{equation}
and moreover that this integral is bounded as $|y|\to\infty$.
 Assume that
 \begin{equation}\label{infinity}
\lim_{y\to\infty}|f(x\pm iy)|e^{-2\pi y}=0,
\end{equation}
uniformly in $x$ on every finite interval.
Also assume that
\begin{equation}\label{rlimit}
\lim_{r\to\infty}\int_{-\infty}^{\infty}e^{-2\pi |y|}|f(r+iy)|dy=0.
\end{equation}
  Recall that the sequence $\lambda_n$, $n\geq1$, is defined in Lemma \ref{rooty}. Assume that
$$\sum_{n=1}^{\infty}f(\lambda_n)$$ converges.
Also, recall that $\sigma(t)$ is defined in \eqref{sig}.
 Then
\begin{align}\label{main}
\sum_{n=1}^{\infty}f(\lambda_n)=&-\df12 f(0)+\int_0^{\infty}\df{p\left(p+\df{1}{\pi}\right)+x^2}{p^2+x^2}f(x)dx\notag\\
& -\df{1}{2\pi}\int_0^{\infty}\{f^{\prime}(it)+f^{\prime}(-it)\}
\log\left(\df{1}{1-\sigma(-t)e^{-2\pi t}}\right)dt.
\end{align}
\end{theorem}

This theorem is especially interesting, because the summands $f(\lambda_n)$ depend in a very unusual way on the parameter $p$.  Apart from Koshliakov's $p$-analogue of the Poisson summation formula \cite[p.~58, Equation (V)]{koshliakov3}, we know of no other theorem of this sort in the literature.

\section{Proof of Theorem \ref{theorem1}}
\begin{proof}
Now, recalling \eqref{sig}, we have
$$\df{d}{dz}\sigma(iz)=\df{i}{p-iz}+i\df{p+iz}{(p-iz)^2}=\df{2pi}{(p-iz)^2}.$$
Thus, if $G(z):=\sigma(iz)e^{2\pi iz}-1$,
\begin{align}
G^{\prime}(z)=&\df{2pi}{(p-iz)^2}e^{2\pi iz}+2\pi i\df{p+iz}{p-iz}e^{2\pi iz}\notag\\
=&\df{2pi+2\pi i(p^2+z^2)}{(p-iz)^2}e^{2\pi iz}.\label{G}
\end{align}

We now integrate
\begin{equation}\label{F}
F(z):=\df{p\left(p+\df{1}{\pi}\right)+z^2}{p^2+z^2}\df{f(z)}{\sigma(iz)e^{2\pi iz}-1}
\end{equation}
over an indented rectangle  denoted by $\mathcal{C}_r$ and defined as follows.  The horizontal sides are given by $z=x\pm iN$, where $0\leq x\leq r$,  $N$ is a positive number, and $r$ is a positive integer. The vertical sides pass through the origin and the real point $z=r$. There is a semi-circular indentation around  $z=ip$, denoted by $C_{p}$, lying in the right half-plane, and with radius $\epsilon<\tf12$.
Note that at $z=-ip$, $1/\{\sigma(iz)e^{2\pi iz}-1\}=0$.  Thus, $F(z)$ is analytic at $z=-ip$.
The contour $\mathcal{C}_r$ also contains a semi-circular portion around $0$, which lies in the right half-plane, which has  radius $\epsilon<\tf12$, and which is denoted by $C_0=:C_{01}\cup C_{02}$, where $C_{01}$ and $C_{02}$ are those parts of $C_0$ in $\mathbb{H}$ and $\mathbb{H}^-$, respectively.  Furthermore, define $\Gamma_r$ by
\begin{align}\label{contours}
\mathcal{C}_r:=&\Gamma_r+C_{p}+C_0\notag\\
:=&\Gamma_{r1}+\Gamma_{r2}+C_{p}+C_0,
\end{align}
where $\Gamma_{r1}$ and $\Gamma_{r2}$ are those portions of $\Gamma_r$ in $\mathbb{H}$ and $\mathbb{H}^-$, respectively.
Furthermore, let
$$\Gamma_{r11}=\Gamma_{r1}\cup [\epsilon,r]\qquad\text{and} \qquad \Gamma_{r22}=\Gamma_{r2}\cup [r,\epsilon].$$
By Lemma \ref{rooty}, the function $F(z)$ has simple poles on the interior of $\mathcal{C}_r$ at $\lambda_1, \lambda_2,\dots,\lambda_r$. (Note that $\lambda_k<k$ for $1\leq k\leq r$.)

Let $R_a$ denote the residue of a typical pole of $F(z)$ at $z=a$.    Recall from Lemma \ref{rooty} that a pole $a$ of $F(z)$ on the interior of $\mathcal{C}_r$ occurs  when  $e^{2\pi ia}=1/\sigma(ia)$, where $\sigma(t)$ is defined in \eqref{sig}.  Hence, employing  \eqref{G} and \eqref{F} and using a basic formula for the residue of a function with a simple pole, we find that
\begin{align}\label{residue}
R_a=& \df{p\left(p+\df{1}{\pi}\right)+a^2}{p^2+a^2}f(a)\df{(p-ia)^2}{2pi+2\pi i(p^2+a^2)}\df{p+ia}{p-ia}\notag\\
=&\df{p\left(p+\df{1}{\pi}\right)+a^2}{2pi+2\pi i(p^2+a^2)}f(a)\notag\\
=&\df{1}{2\pi i}\df{p\left(p+\df{1}{\pi}\right)+a^2}{\df{p}{\pi}+(p^2+a^2)}f(a)\notag\\
=&\df{f(a)}{2\pi i}.
\end{align}
Thus, by the residue theorem,
\begin{equation}\label{residuetheorem}
\int_{\mathcal{C}_r} F(z)dz=\int_{\Gamma_r}F(z)dz+\int_{C_p}F(z)+\int_{C_{0}}F(z)dz=
\sum_{n=1}^r f(\lambda_n).
\end{equation}

Let $\mu(z)$ be defined by
\begin{equation}\label{mu2}
\mu(z):=\df{p\left(p+\df{1}{\pi}\right)+z^2}{p^2+z^2}f(z).
\end{equation}
Observe that
\begin{equation}\label{mu1}
F(z)=-\mu(z)-\df{\mu(z)}{\sigma(-iz)e^{-2\pi iz}-1}.
\end{equation}
Also observe that $\mu(z)$ is analytic for $\Re(z)\geq0$, except for $z=\pm ip$.
By Cauchy's Theorem,
\begin{equation*}
\int_{\Gamma_{r1}}\mu(z)dz
+\int_{\epsilon}^r\mu(x)dx+\int_{C_{p}}\mu(z)dz+\int_{C_{01}}\mu(z)dz=0,
\end{equation*}
or
\begin{equation}\label{2}
\int_{\Gamma_{r1}}\mu(z)dz=-\int_{\epsilon}^r\mu(x)dx-\int_{C_{p}}\mu(z)dz-\int_{C_{01}}\mu(z)dz.
\end{equation}
Thus, by \eqref{residuetheorem} and \eqref{2},
\begin{align}\label{3}
\int_{\mathcal{C}_r} F(z)dz=&\int_{\Gamma_r}F(z)dz+\int_{C_p}F(z)dz+\int_{C_{0}}F(z)dz\notag\\
=&\int_{\Gamma_{r1}}F(z)dz+\int_{\Gamma_{r2}}F(z)dz+\int_{C_p}F(z)dz+\int_{C_{0}}F(z)dz\notag\\
=&-\int_{\Gamma_{r1}}\mu(z)dz-\int_{\Gamma_{r1}}\df{\mu(z)dz}{\sigma(-iz)e^{-2\pi iz}-1}\notag\\
&+\int_{\Gamma_{r2}}F(z)dz+\int_{C_p}F(z)dz+\int_{C_{0}}F(z)dz\notag\\
=&\int_{\epsilon}^r\mu(x)dx+\int_{C_{01}}\mu(z)dz+\int_{C_{p}}\mu(z)dz
-\int_{\Gamma_{r1}}\df{\mu(z)dz}{\sigma(-iz)e^{-2\pi iz}-1}\notag\\
&+\int_{\Gamma_{r2}}F(z)dz+\int_{C_p}F(z)dz+\int_{C_{0}}F(z)dz\notag\\
=&\int_{C_0}F(z)dz+\int_{\epsilon}^r\mu(x)dx+\int_{C_{01}}\mu(z)dz+\int_{C_{p}}\left\{F(z)+\mu(z)\right\}dz\notag\\
&-\int_{\Gamma_{r1}}\df{\mu(z)dz}{\sigma(-iz)e^{-2\pi iz}-1}+
\int_{\Gamma_{r2}}\df{\mu(z)dz}{\sigma(iz)e^{2\pi iz}-1},
\end{align}
by \eqref{F}.

 We examine each of the integrals in \eqref{3}.

First, recall from \eqref{residue} that $R_0=f(0)/(2\pi i)$.  Thus,
\begin{gather}\label{5}
\lim_{\epsilon\to0}\int_{C_0}F(z)dz=\lim_{\epsilon\to0}\int_{C_0}\left(\df{f(0)}{2\pi i}\df{1}{z}+\cdots\right)
=\lim_{\epsilon\to0}\df{f(0)}{2\pi i}\int_{\pi/2}^{-\pi/2}\df{i\epsilon e^{i\theta}}{\epsilon e^{i\theta}}d\theta\notag\\
=\df{f(0)}{2\pi}\int_{\pi/2}^{-\pi/2}d\theta=-\df12 f(0).
\end{gather}

Second, because $\mu(z)$ is analytic at $z=0$,
\begin{equation}\label{6}
\lim_{\epsilon\to0}\int_{C_{01}}\mu(z)dz=0.
\end{equation}

Third, we note that at $z=ip$, $1/\{\sigma(-iz)e^{-2\pi iz}-1\}=0$.  Thus, $F(z)+\mu(z)$ is analytic at $z=ip$, and so
\begin{equation}\label{7}
\lim_{\epsilon\to 0}\int_{C_{p}}\left\{F(z)+\mu(z)\right\}dz=0.
\end{equation}

Fourth, we examine the contributions of each of the latter two integrals in \eqref{3} on the imaginary axis.  In the first integral below, we set $z=iy$, and in the second integral we set $z=-iy$.  Thus, we find that
\begin{align}\label{4}
&-\int_{[iN,0]}\df{\mu(z)dz}{\sigma(-iz)e^{-2\pi iz}-1}+
\int_{[0,-iN]}\df{\mu(z)dz}{\sigma(iz)e^{2\pi iz}-1}\notag\\
=&i\int_0^N\df{\mu(iy)dy}{\sigma(y)e^{2\pi y}-1}-i\int_0^N\df{\mu(-iy)dy}{\sigma(y)e^{2\pi y}-1}\notag\\
=&i\int_0^N\df{\left\{\mu(iy)-\mu(-iy)\right\}dy}{\sigma(y)e^{2\pi y}-1}.
\end{align}
Next, we examine these integrals on the top side of $\mathcal{C}_r$.
Set $z=x+iN, 0\leq x \leq r$. By our hypotheses \eqref{hypoth} and \eqref{infinity}, there exist positive constants $C_1$ and $C_2$ such that
\begin{align}\label{14}
\left|\int_{[iN,iN+r]}\df{\mu(z)dz}{\sigma(-iz)e^{-2\pi iz}-1}\right|
&\leq C_1\int_{0}^r|\mu(z)|e^{-2\pi N}dx\leq C_2e^{-2\pi N}\int_0^{r}|f(x+iN)|dx\notag\\
&\leq C_2e^{-2\pi N}\int_0^{\infty}|f(x+iN)|dx=o(1),
\end{align}
as $N\to\infty$.
 By a similar argument,
\begin{equation}\label{15}
\left|\int_{[-iN,-iN+r]}\df{\mu(z)dz}{\sigma(iz)e^{2\pi iz}-1}\right|=o(1),
\end{equation}
as $N\to\infty$.
Lastly, after we have let $N\to\infty$, there remain the examinations of the latter two integrals  on the right side of \eqref{3}.
By \eqref{rlimit}, there exist constants $C_3$ and $C_4$ such that
\begin{equation}\label{16}
\left|\int_{0}^{\infty}\df{\mu(r+iy)dy}{\sigma(-ir+y)e^{2\pi (-ir+y)}-1}\right|
\leq C_3\int_{0}^{\infty}e^{-2\pi |y|}|f(r+iy)|dy=o(1),
\end{equation}
as $r\to\infty$.
Similarly,
\begin{equation}\label{16a}
\left|\int_{-\infty}^{0}\df{\mu(r+iy)dy}{\sigma(ir-y)e^{2\pi (ir-y)}-1}\right|
\leq C_4\int_{-\infty}^{0}e^{-2\pi |y|}|f(r+iy)|dy=o(1),
\end{equation}
as $r\to\infty$.

Hence,  letting $\epsilon\to0$, $r\to\infty$, $N\to\infty$, and using \eqref{4}--\eqref{16a},  we find that \eqref{3} yields
\begin{align}\label{9}
\lim_{r\to\infty}\int_{\mathcal{C}_r} F(z)dz=-\df12 f(0)+\int_0^{\infty}\mu(x)dx+
i\int_0^{\infty}\df{\left\{\mu(iy)-\mu(-iy)\right\}dy}{\sigma(y)e^{2\pi y}-1}.
\end{align}

We next give an alternative representation for the integral on the far right side of \eqref{9}.  To that end,
\begin{align}\label{10}
\df{d}{dt}\log\{1-\sigma(-t)e^{-2\pi t}\}
=&\df{1}{1-\sigma(-t)e^{-2\pi t}}\left(\sigma^{\prime}(-t)e^{-2\pi t}+2\pi \sigma(-t)e^{-2\pi t}\right)\notag\\
=&\df{e^{-2\pi t}}{1-\sigma(-t)e^{-2\pi t}}\left(\df{2p}{(p+t)^2}+2\pi\df{p-t}{p+t}\right)\notag\\
=&\df{1}{e^{2\pi t}-\sigma(-t)}\df{2p+2\pi(p^2-t^2)}{(p+t)^2}\notag\\
=&\df{1}{e^{2\pi t}-(p-t)/(p+t)}\df{2p+2\pi(p^2-t^2)}{(p+t)^2}\notag\\
=&\df{(p+t)/(p-t)}{\sigma(t)e^{2\pi t}-1}\,\df{2p+2\pi(p^2-t^2)}{(p+t)^2}\notag\\
=&\df{1}{\sigma(t)e^{2\pi t}-1}\left(\df{2p}{p^2-t^2}+2\pi\right),
\end{align}
or, upon rewriting \eqref{10}, we deduce that
\begin{align}\label{11}
\df{1}{2\pi}\df{d}{dt}\log\{1-\sigma(-t)e^{-2\pi t}\}
=&\df{1}{\sigma(t)e^{2\pi t}-1}\left(\df{p}{\pi(p^2-t^2)}+1\right)\notag\\
=&\df{1}{\sigma(t)e^{2\pi t}-1}\df{p\left(p+\df{1}{\pi}\right)-t^2}{p^2-t^2}.
\end{align}
Hence, using \eqref{11} in \eqref{9} and  integrating by parts, we conclude that
\begin{align}\label{12}
&i\int_0^\infty\df{\left\{\mu(it)-\mu(-it)\right\}}{\sigma(t)e^{2\pi t}-1}dt\notag\\
=&i\int_0^{\infty}\left(\df{p\left(p+\df{1}{\pi}\right)-t^2}{p^2-t^2}f(it)-
\df{p\left(p+\df{1}{\pi}\right)-t^2}{p^2-t^2}f(-it)\right)\df{dt}{\sigma(t)e^{2\pi t}-1}\notag\\
=&i\int_0^{\infty}\left\{f(it)-f(-it)\right\}\df{1}{2\pi}\df{d}{dt}\log\{1-\sigma(-t)e^{-2\pi t}\}dt\notag\\
=&\df{i}{2\pi}\biggl(\left\{f(it)-f(-it)\right\}   \log(1-\sigma(-t)e^{-2\pi t})\biggr|_0^\infty    \notag\\
&-i\int_0^{\infty}\left\{f^{\prime}(it)+f^{\prime}(-it)\right\}    \log(1-\sigma(-t)e^{-2\pi t})dt\biggr)\notag\\
=&\df{1}{2\pi}\int_0^{\infty}\left\{f^{\prime}(it)+f^{\prime}(-it)
\right\}\log(1-\sigma(-t)e^{-2\pi t})dt,
\end{align}
where we appealed to \eqref{infinity}.
Employing \eqref{12} in \eqref{9}, we deduce that
\begin{align}\label{13}
\lim_{r\to\infty}\int_{\mathcal{C}_r} F(z)dz=&-\df12 f(0)+\int_0^{\infty}\mu(x)dx\notag\\&+
\df{1}{2\pi}\int_0^{\infty}\left\{f^{\prime}(it)+f^{\prime}(-it)\right\}\log(1-\sigma(-t)e^{-2\pi t})dt.
\end{align}

  Combining \eqref{residuetheorem} with \eqref{13} and recalling the definition of $\mu(z)$ in \eqref{mu2},  we complete the proof of Theorem \ref{theorem1}.
  \end{proof}

\section{Analogues of the Abel--Plana Summation Formula}
We derive an alternative version of Theorem \ref{theorem1} which shows a clearer connection with the Abel--Plana summation formula.

\begin{theorem}\label{theorem3}  Assume the hypotheses of Theorem \ref{theorem1}.  Then
\begin{align*}
 \sum_{n=1}^{\infty}f(\lambda_n)=&-\df12 f(0)+\int_0^{\infty}\df{p\left(p+\df{1}{\pi}\right)+x^2}{p^2+x^2}f(x)dx\notag\\
  &+i\int_0^{\infty}\{f(it)-f(-it)\}\df{\sigma(-t)+\sigma^{\prime}(-t)/(2\pi)}{e^{2\pi t}-\sigma(-t)}dt.
  \end{align*}
\end{theorem}

\begin{proof} Integrating by parts, we find that
\begin{align}\label{I}
I:=&\int_0^{\infty}\{f^{\prime}(it)+f^{\prime}(-it)\}\log\left(\df{1}{1-\sigma(-t)e^{-2\pi t}}\right)dt\notag\\
=&\df{1}{i}\{f(it)-f(-it)\}\log\left(\df{1}{1-\sigma(-t)e^{-2\pi t}}\right)\bigg|_0^{\infty}\notag\\
&-\df{1}{i}\int_0^{\infty}\{f(it)-f(-it)\}
\df{(-2\pi\sigma(-t)-\sigma^{\prime}(-t))e^{-2\pi t}}{1-\sigma(-t)e^{-2\pi t}}dt.
\end{align}
We show that the integrated term vanishes.  Consider first the limit as $t\to0$.  Since
 $f(z)$ is analytic at $z=0$,
\begin{equation}\label{relax}
f(it)-f(-it)=O(t),
\end{equation}
as $t\to 0$.  Also, as $t\to 0$, $\sigma(-t)\to 1$, and so $\log\left(\df{1}{1-\sigma(-t)e^{-2\pi t}}\right)$ has a logarithmic singularity at $t=0$. Hence,
\begin{equation}\label{relax1}
\lim_{t\to 0}\df{1}{i}\{f(it)-f(-it)\}\log\left(\df{1}{1-\sigma(-t)e^{-2\pi t}}\right)=0.
\end{equation}
Secondly, consider the limit as $t\to\infty$.  As $t\to\infty$, $\sigma(-t)\to -1$.
Also, $\log(1+x) \sim x$ as $x$ tends to $0$.  Thus, as $t\to\infty$,
\begin{equation}\label{100}
\log\left(1-\sigma(-t)e^{-2\pi t}\right)\to e^{-2\pi t}.
\end{equation}
Thus, from  \eqref{infinity} and \eqref{100}, we conclude that
 \begin{equation}\label{infinity1}
\lim_{t\to\infty}|f(\pm it)|\log\left(1-\sigma(-t)e^{-2\pi t}\right)=0.
\end{equation}
Thus, the first term on the right-hand side of \eqref{I} is indeed equal to 0. Hence, we have shown that
\begin{gather}
 -\df{1}{2\pi}\int_0^{\infty}\{f^{\prime}(it)+f^{\prime}(-it)\}\log\left(\df{1}{1-\sigma(-t)e^{-2\pi t}}\right)dt\notag\\
 =i\int_0^{\infty}\{f(it)-f(-it)\}\df{\sigma(-t)+\sigma^{\prime}(-t)/(2\pi)}{e^{2\pi t}-\sigma(-t)}dt.\label{finalintegral}
 \end{gather}
 Finally, from \eqref{main} and \eqref{finalintegral}, we conclude that
 \begin{align*}
 \sum_{n=1}^{\infty}f(\lambda_n)=&-\df{1}{2} f(0)+\int_0^{\infty}\df{p\left(p+\df{1}{\pi}\right)+x^2}{p^2+x^2}f(x)dx\notag\\
  &+i\int_0^{\infty}\{f(it)-f(-it)\}\df{\sigma(-t)+\sigma^{\prime}(-t)/(2\pi)}{e^{2\pi t}-\sigma(-t)}dt.
  \end{align*}
\end{proof}

If in Theorem \ref{theorem3} we let $p\to\infty$, then $\sigma(-t)\to 1$ and $\sigma^{\prime}(-t)\to 0$. Also $\lambda_n\to n$, $n\geq1$.  We thus immediately obtain the Abel--Plana summation formula \eqref{abel}.

As mentioned in the introduction, Koshliakov \cite[Chapter 3, Equation (I)]{koshliakov3} derived an equivalent form of his generalization of the Abel--Plana summation formula, which is given in Theorem \ref{theorem1}. This equivalent form is given next.

\begin{theorem}\label{theorem4} Under the same hypotheses of Theorem \ref{theorem1},
\begin{align}\label{koshman1}
\sum_{j=1}^{\infty}\frac{p^2 + \l_{j}^2}{p(p+\frac{1}{\pi})+\l^2_{j}}f(\l_j)&=-\frac{1}{2}\frac{1}{1+\frac{1}{\pi p}}f(0) +\int_{0}^{\infty}f(x)dx \nonumber\\
&+i\int_{0}^{\infty}\{f(i x)-f(-i x)\}\frac{dx}{\s(x)e^{2\pi x}-1}.
\end{align}
\end{theorem}

\begin{proof}
As we shall show, Theorem \ref{theorem4} is easily derived from  Theorem \ref{theorem3}. To see this, replace $f(x)$   by $\displaystyle\frac{p^{2}+x^2}{p\left(p+\frac{1}{\pi}\right)+x^2}f(x)$ in Theorem \ref{theorem3}. All of the expressions in \eqref{koshman1}, except the last expression on the right-hand side of \eqref{koshman1}, can be straightforwardly obtained. As for the last one, observe that
\begin{align*}
&i\int_{0}^{\infty}\frac{(p^2-t^2)\{f(it)-f(-it)\}}{p
\left(p+\frac{1}{\pi}\right)-t^2}\left\{\frac{\frac{p-t}{p+t}-\frac{1}{2\pi}
\frac{d}{dt}\left(\frac{p-t}{p+t}\right)}{e^{2\pi t}-\frac{p-t}{p+t}}\right\}\, dt  \nonumber\\
&=i\int_{0}^{\infty}\frac{(p^2-t^2)\{f(it)-f(-it)\}}{p\left(p+\frac{1}{\pi}\right)-t^2}
\left\{\frac{1+\frac{p}{\pi(p+t)(p-t)}}{\sigma(t)e^{2\pi t}-1}\right\}\, dt\nonumber\\
&=i\int_{0}^{\infty}\frac{f(it)-f(-it)}{\sigma(t)e^{2\pi t}-1}\, dt.
\end{align*}
This completes our proof of Theorem \ref{theorem4}.
\end{proof}

\section{Three Analogues of the Abel--Plana Summation Formula Due to Ramanujan}

On pages 334 and 335 in the unorganized pages of his second notebook \cite{nb}, Ramanujan offers the Abel--Plana summation formula and four analogues  \cite[pp.~411--414]{V}.  We state three of them  and then show that they can be deduced from Koshliakov's summation formulas.

\begin{entry}\label{entry4} Let $\varphi(z)$ be analytic for $\Re\,z\geq 0$. Assume that
\begin{equation}\label{entry4a}
\lim_{y\to\infty}|\varphi(x\pm iy)|e^{-\frac{\pi y}{2}}=0,
\end{equation}
uniformly for $x$ in any compact interval on $[0,\infty).$  Suppose also that
\begin{equation}\label{entry4b}
\int_0^{\infty}|\varphi(x\pm iy)|e^{-\frac{\pi y}{2}}dy
\end{equation}
exists for all $x\geq 0$ and tends to $0$ as $x$ tends to $\infty$.  Assume that the integral below exists.  Then
\begin{equation}\label{4c}
4\sum_{n=0}^{\infty}(-1)^n\varphi(2n+1)=\int_0^{\infty}
\df{\varphi(ix)+\varphi(-ix)}{\cosh(\pi x/2)}dx.
\end{equation}
\end{entry}

It should be remarked that in his notebooks \cite{nb}, Ramanujan states Entries \ref{entry4}--\ref{entry6} with no attached hypotheses.

\begin{entry}\label{entry5}
Let $\varphi(z)$ be analytic for $\Re\,z\geq 0$. Assume that
\begin{equation}\label{entry5a}
\lim_{y\to\infty}|\varphi(x\pm iy)|e^{-\frac{\pi y}{2}}=0,
\end{equation}
uniformly for $x$ in any compact interval on $[0,\infty).$  Suppose also that
\begin{equation}\label{entry5b}
\int_0^{\infty}|\varphi(x\pm iy)|e^{-\frac{\pi y}{2}}dy
\end{equation}
exists for all $x\geq 0$ and tends to $0$ as $x$ tends to $\infty$.  Assume that the integral below exists.  Then
\begin{equation}\label{5c}
\frac12\varphi(0)+\sum_{n=1}^{\infty}(-1)^n\varphi(n)=\df{i}{2}
\int_0^{\infty}\df{\varphi(ix)-\varphi(-ix)}{\sinh(\pi x)}dx.
\end{equation}
\end{entry}

\begin{entry}\label{entry6}
Let $\varphi(z)$ be analytic for $\Re\,z\geq 0$. Suppose that
\begin{equation}\label{entry6a}
\lim_{y\to\infty}|\varphi(x\pm iy)|e^{-\pi y}=0,
\end{equation}
uniformly for $x$ in any compact interval on $[0,\infty).$  Assume also that
\begin{equation}\label{entry6b}
\int_0^{\infty}|\varphi(x\pm iy)|e^{-\pi y}dy
\end{equation}
exists for all $x\geq 0$ and tends to $0$ as $x$ tends to $\infty$.  Then, provided that the integral below exists,
\begin{equation}\label{6c}
\sum_{n=0}^{\infty}\varphi(2n+1)=\df{1}{2}
\int_0^{\infty}\varphi(x)dx
-\df12 i\int_0^{\infty}\df{\varphi(ix)-\varphi(-ix)}{e^{\pi x}+1}dx.
\end{equation}
\end{entry}

We now show that Entry \ref{entry6} can be derived from Theorem \ref{theorem4}.

\begin{proof}
If we let $p \to 0$ in \eqref{koshman1}, then $\lambda_n \to n-1/2$, and we have
\begin{align}\label{kosh1}
    \sum_{n=1}^{\infty}f\left( n-\frac{1}{2}\right)&=\int_{0}^{\infty}f(x)dx -i\int_{0}^{\infty}\frac{f(i x)-f(-i x)}{e^{2\pi x}+1}dx.
\end{align}
Now, if we let $f(x)=\varphi(2x)$ in \eqref{kosh1}, then we obtain  \eqref{6c}.
\end{proof}

Entries \ref{entry4} and \ref{entry5} can be derived from another analogue of the Abel--Plana summation formula that is due to Koshliakov, namely, Theorem \ref{koshagain} below.  To prove that theorem, we require the following lemma.

\begin{lemma}\label{L1} Let $\varphi(z)$ be analytic in the quarter-plane, $\Re(z)\geq0, \Im(z)\geq0$.  Assume that
\begin{equation}\label{a}
\lim_{y\to\infty}|\varphi(x\pm iy)|e^{-\pi y}=0,
\end{equation}
uniformly in $x$ on every finite interval, and that
\begin{equation}\label{b}
\int_0^{\infty}|\varphi(x\pm iy)|e^{-\pi y}dy
\end{equation}
exists for every $x\geq 0$ and tends to $0$ as $x\to\infty$.  Then
\begin{equation}\label{c}
\int_0^{\infty}e^{\pi ix}\varphi(x)dx=i\int_0^{\infty}e^{-\pi x}\varphi(ix)dx.
\end{equation}
\end{lemma}

\begin{proof} Let $C_{r,N}$ denote the rectangle with $0\leq x\leq r, 0\leq y\leq N$.  By Cauchy's theorem,
\begin{equation}\label{d}
\int_{C_{r,N}}e^{\pi iz}\varphi(z)dz=0.
\end{equation}
We first examine the integral over the upper side of $C_{r,N}$. To that end,
\begin{equation}\label{e}
\left|\int_0^{r}e^{\pi i(x+iN)}\varphi(x+iN)dx\right|
\leq\int_0^r e^{-\pi N}|\varphi(x+iN)|dx =o(1),
\end{equation}
as $N\to\infty$,  by \eqref{a}.  Second, we examine the integral over the right side of $C_{r,N}$.  Hence, as $r\to\infty$,
\begin{equation}\label{f}
\left|\int_0^{\infty}e^{\pi i(r+iy)}\varphi(r+iy)i\,dy\right|\leq
\int_0^{\infty}e^{-\pi y}|\varphi(r+iy)|dy\,\,\to 0,
\end{equation}
  by our hypothesis \eqref{b}.  Hence, by \eqref{d}--\eqref{f},
  \begin{equation}\label{g}
  \int_0^{\infty}e^{\pi ix}\varphi(x)dx+\int_{\infty}^0e^{\pi i(0+iy)}\varphi(0+iy)i\,dy=0,
  \end{equation}
  or
  \begin{equation*}
  \int_0^{\infty}e^{\pi ix}\varphi(x)dx-i\int_0^{\infty}e^{-\pi y}\varphi(iy)dy =0,
  \end{equation*}
  which is equivalent to \eqref{c}.
\end{proof}

\begin{theorem}\label{koshagain}
Assume the hypotheses of Theorem \ref{theorem1}. Then, for $p>0$,
\begin{align}\label{koshagain1}
   \sum_{n=1}^{\infty}&\frac{p^2 + \l_{n}^2}{p(p+\frac{1}{\pi})+\l^2_{n}}e^{\pi i \l_n}\phi(2\l_n)=-\frac{1}{2}\frac{1}{1+\frac{1}{\pi p}}\phi(0)+\frac{i}{2}\int_{0}^{\infty}\frac{\s(x/2)\phi(i x)-\phi(-ix)}{\s(x/2)e^{\pi x/2}-e^{-\pi x/2}}dx.
\end{align}
\end{theorem}

If we let $p \to 0$ in \eqref{koshagain1}, we obtain
\begin{align*}
4\sum_{n=0}^{\infty}(-1)^n\phi(2n+1)=\int_{0}^{\infty}\frac{\phi(ix)+\phi(-ix)}{\cosh{(\pi x/2)}}dx,
\end{align*}
which is Entry \ref{entry4}.

If we let $p \to \infty$ in \eqref{koshagain1}, we find that
\begin{align*}
\sum_{n=1}^{\infty}(-1)^n\phi(2n)=-\frac{1}{2}\phi(0)
+\frac{i}{4}\int_{0}^{\infty}\frac{\phi(ix)-\phi(-ix)}{\sinh\left(\frac{\pi x}{2}\right)}dx,
\end{align*}
which is Entry \ref{entry5}, upon letting $\varphi(x)=\phi(2x)$ and simplifying.

We now prove Theorem \ref{koshagain}.

\begin{proof}
Letting $f(z)=e^{\pi i z}\phi(2z)$ in \eqref{koshman1}, we  find that \begin{align}
   \sum_{n=1}^{\infty}&\frac{p^2 + \l_{n}^2}{p(p+\frac{1}{\pi})+\l^2_{n}}e^{\pi i \l_n}\phi(2\l_n)\nonumber\\
   =&-\frac{1}{2}\frac{1}{1+\frac{1}{\pi p}}\phi(0) +\frac{1}{2}\int_{0}^{\infty}e^{\pi i x/2}\phi(x)dx
+i\int_{0}^{\infty}\frac{e^{-\pi x}\phi(2ix)-e^{\pi x}\phi(-2ix)}{\s(x)e^{2\pi x}-1}dx\nonumber\\
=&-\frac{1}{2}\frac{1}{1+\frac{1}{\pi p}}\phi(0) +\frac{1}{2}\int_{0}^{\infty}e^{\pi i x/2}\phi(x)dx
-\frac{i}{2}\int_{0}^{\infty}\frac{e^{\pi x/2}\phi(-ix)-e^{-\pi x/2}\phi(ix)}{e^{\pi x/2}\left(\s(x/2)e^{\pi x/2}-e^{-\pi x/2}\right)}dx\nonumber\\
=&-\frac{1}{2}\frac{1}{1+\frac{1}{\pi p}}\phi(0) +\frac{1}{2}\int_{0}^{\infty}e^{\pi i x/2}\phi(x)dx \nonumber\\
&-\frac{i}{2}\int_{0}^{\infty}\frac{e^{\pi x/2}\phi(-ix)-\s(x/2)e^{\pi x/2}\phi(i x)-e^{-\pi x/2}\phi(ix)+\s(x/2)e^{\pi x/2}\phi(i x)}{e^{\pi x/2}\left(\s(x/2)e^{\pi x/2}-e^{-\pi x/2}\right)}dx\nonumber\\
=&-\frac{1}{2}\frac{1}{1+\frac{1}{\pi p}}\phi(0) +\frac{1}{2}\int_{0}^{\infty}e^{\pi i x/2}\phi(x)dx \nonumber\\
&-\frac{i}{2}\int_{0}^{\infty}\frac{\phi(-ix)-\s(x/2)\phi(i x)}{\s(x/2)e^{\pi x/2}-e^{-\pi x/2}}dx-\frac{i}{2}\int_{0}^{\infty}\frac{\s(x/2)e^{\pi x/2}\phi(i x)-e^{-\pi x/2}\phi(ix)}{e^{\pi x/2}\left(\s(x/2)e^{\pi x/2}-e^{-\pi x/2}\right)}dx\nonumber\\
=&-\frac{1}{2}\frac{1}{1+\frac{1}{\pi p}}\phi(0) +\frac{1}{2}\int_{0}^{\infty}e^{\pi i x/2}\phi(x)dx \nonumber\\
&-\frac{i}{2}\int_{0}^{\infty}\frac{\phi(-ix)-\s(x/2)\phi(i x)}{\s(x/2)e^{\pi x/2}-e^{-\pi x/2}}dx-\frac{i}{2}\int_{0}^{\infty}e^{-\pi x/2}\phi(i x)dx.\nonumber\\
=&-\frac{1}{2}\frac{1}{1+\frac{1}{\pi p}}\phi(0)-\frac{i}{2}\int_{0}^{\infty}\frac{\phi(-ix)-\s(x/2)\phi(i x)}{\s(x/2)e^{\pi x/2}-e^{-\pi x/2}}dx,
\end{align}
by an application of Lemma \ref{L1} with $\varphi(x)$ replaced by $\phi(2x)$. Our proof is therefore complete.
\end{proof}

\section{Two Further Extensions of the Abel--Plana Summation Formula}

We state two further extensions of the Abel--Plana summation formula, one due to Koshliakov, and the other due to Ramanujan.  We do not provide proofs here.

In Chapter 6 of \cite[p.~101, Equation (I)]{koshliakov3}, Koshliakov derived another extension of \eqref{abel}.

\begin{theorem}\label{koshalt}
Let $s=\sigma+it$, and $\a, \b$ be two real numbers such that $-1< \a <0, ~\b>0$. Let $F(s)$ be a holomorphic function of $s$ in $\a<\s<\b+1$, and such that
$F(s)=\mathcal{O}\left(|t|^{-\l} \right),~ \l>1$,
uniformly in $\a<\s<\b+1$ as $|t|\to\infty$. Define
\begin{align}\label{7.4}
f_1(y):=\frac{1}{2\pi i}\int_{\a- i \infty}^{\a+i \infty}F(s)\frac{ds}{y^s}, ~ y>0.
\end{align}
Furthermore, define \begin{align*}
f(x):= \int_{0}^{\infty}e^{-xy}f_1(y)dy.
\end{align*}
Then, for $p \in \mathbb{R}^{+}$,
\begin{align*}
\sum_{n=1}^{\infty}\int_{0}^{\infty}e^{-nx}\left(\frac{2\pi p -x}{2\pi p +x} \right)^nf_1(x)~dx =&-\frac{1}{2}f(0)+\frac{1}{1+\frac{1}{\pi p}}\int_{0}^{\infty}f(x)dx\\
&-2\int_{0}^{\infty}\frac{f(ix)-f(-ix)}{2i}\s_p(2 \pi x)dx,
\end{align*}
where \cite[Chapter 2, p.~44, Equation (33)]{koshliakov3}
\begin{align}\label{3.33}
\s_p(z):=\sum_{j=1}^{\infty}\frac{p^2+\lambda_j^2}{p\left(p+\frac{1}{\pi}\right)+\lambda_j^2}e^{-\l_j z},
\end{align}
and each $\l_j$, $1\leq j<\infty$, is a solution of \eqref{rootsa}.
\end{theorem}

On page 334 in his second notebook, Ramanujan offers another beautiful analogue of the  Abel--Plana summation formula \cite[p.~412]{V}.

\begin{theorem} Assume that $\varphi(z)$ is analytic for $\Re\,z\geq0$.  Let $\alpha$ be real with $0<|\alpha|<1$.  Suppose that
\begin{equation*}
    \lim_{y\to\infty}|\varphi(x\pm iy)|e^{-\pi y}=0,
\end{equation*}
uniformly in $x$ on any finite interval in $x\geq0$. Suppose also that
\begin{equation*}
\int_{-\infty}^{\infty}
\df{\varphi(x+iy)}{\cos(\pi(x+iy))+cos(\pi\alpha)}dy
\end{equation*}
exists for each non-negative number $x$ and tends to $0$ as $x\to\infty$.  Assume also that the integral below converges. Then,
\begin{equation*}
 \df{2}{\sin(\pi\alpha)}\sum_{n=0}^{\infty}
 \{\varphi(2n+1-\alpha)-\varphi(2n+1+\alpha)\}=
 \int_0^{\infty}\df{\varphi(ix)+\varphi(-ix)}
 {\cosh(\pi x)+\cos(\pi\alpha)}dx.
 \end{equation*}
 \end{theorem}

\section{Examples}
In this section, we give both  new and known special cases of our main theorems.

\begin{theorem}\label{theorem5}
Let $p>0$, \textup{Re}$(w)>1,$ and \textup{Re}$(z)>0.$ Then, under the hypotheses of Theorem \ref{theorem1},
\begin{equation*}
\sum_{n=1}^{\infty}\frac{p^2 + \l_{n}^2}{p(p+\frac{1}{\pi})+\l^2_{n}}\l_n^{w-1}e^{-\l_nz}=\Gamma(w)z^{-w}+
2\int_0^{\infty}\df{y^{w-1}\cos\left(\tf12 \pi w-zy\right)}{\sigma(y)e^{2\pi y}-1}\ dy.
\end{equation*}
\end{theorem}

\begin{proof}
Let $f(t)=t^{w-1}e^{-t z}$. Despite the fact that $f(t)$ is not analytic at $t=0$, we can relax the hypotheses around $z=0$ in Theorems \ref{theorem1}, \ref{theorem3}, and \ref{theorem4} to include the present application.   Hence, we have
\begin{align*}
    \sum_{n=1}^{\infty}\frac{p^2 + \l_{n}^2}{p(p+\frac{1}{\pi})+\l^2_{n}}\l_n^{w-1}e^{-\l_nz}&=\int_{0}^{\infty}t^{w-1}e^{-tz}dt +i\int_{0}^{\infty}\frac{(iy)^{w-1}e^{-i y z}-(-iy)^{w-1}e^{i y z}}{\sigma(y)e^{2\pi y}-1}\ dy\\
    &=\Gamma(w)z^{-w}-2\int_{0}^{\infty}\frac{y^{w-1}\sin\left({\frac{\pi}{2}(w-1)-y z}\right)}{\sigma(y)e^{2\pi y}-1}\ dy\\
    &=\Gamma(w)z^{-w}+2\int_{0}^{\infty}\frac{y^{w-1}\cos\left({\frac{\pi w}{2}-y z}\right)}{\sigma(y)e^{2\pi y}-1}\ dy.
\end{align*}
\end{proof}

Letting $p \to \infty$ in Theorem \ref{theorem5}, we obtain an identity found in Ramanujan's notebooks \cite{nb},
 \cite[p.~411, Entry 2]{V}.

\begin{entry}
Let $w$ and $z$ be complex numbers with $\Re\, w>1$ and $\Re\,z>0.$  Then,
\begin{equation*}
\sum_{n=1}^{\infty}n^{w-1}e^{-nz}=\Gamma(w)z^{-w}+
2\int_0^{\infty}\df{y^{w-1}\cos\left(\tf12 \pi w-zy\right)}{e^{2\pi y}-1}dy.
\end{equation*}
 \end{entry}

 However, if we let $p \to 0$ in Theorem \ref{theorem5}, we obtain the following new identity.
\begin{corollary}
Let \textup{Re}$(w)>1$ and \textup{Re}$(z)>0.$ Then,
\begin{equation*}
e^{z/2}\sum_{n=1}^{\infty}\left(n-\frac{1}{2}\right)^{w-1}e^{-nz}=\Gamma(w)z^{-w}-
2\int_0^{\infty}\df{y^{w-1}\cos\left(\tf12 \pi w-zy\right)}{e^{2\pi y}+1}\ dy.
\end{equation*}
\end{corollary}

\section{An Analogue of an Identity in Ramanujan's Second Notebook}

On page 269 in his second notebook \cite{nb}, Ramanujan records the following identity \cite[p.~416]{V}.

\begin{entry}\label{ab}  Let $\alpha, \beta > 0$ with $\alpha\beta=4\pi^2$. Let $\zeta(z)$ denote the Riemann zeta function. If $\Re(n) >0$,  then
\begin{align*}
        &\sqrt{\alpha^n}\bigg(\df{\Gamma(n)\zeta(n)}{(2\pi)^n}+
    \cos\left(\frac12 \pi n\right)
    \sum_{k=1}^{\infty}\df{k^{n-1}}{e^{\alpha k}-1}\bigg)\\
    =&\sqrt{\beta^n}\bigg(\cos\left(\frac12 \pi n\right)\df{\Gamma(n)\zeta(n)}{(2\pi)^n}+
    \sum_{k=1}^{\infty}\df{k^{n-1}}{e^{\beta k}-1}\\
    &-\sin\left(\frac12 \pi n\right)\textup{PV}
    \int_0^{\infty}\df{x^{n-1}}{e^{2\pi x}-1}
    \cot\left(\frac12 \beta x\right)dx\bigg).
\end{align*}
\end{entry}

Entry \ref{ab} is a remarkable theorem, because if we let $n$ be a positive even integer, then Entry \ref{ab} reduces to the well-known transformation formula for Eisenstein series on $\textup{SL}_{2}(\mathbb{Z})$. That there exists a transformation formula when the powers $k^{n-1}$, $\Re(n)>2$, in the summands are \emph{arbitrary} is very surprising. This is another example of a truly \emph{remarkable insight} that possibly only Ramanujan could have made.

The following theorem is a generalization of Entry \ref{ab} in the $p$-setting.
\begin{theorem}\label{thm1}
Let $\eta_p(s)$ be defined in \eqref{eps}. For $\a \b =4\pi^2$ with $\a, \b >0$ and \textup{Re}$(n)>2$,
\begin{align}
    &\sqrt{\a^n}\left(\frac{\Gamma(n)\eta_p(n)}{(2\pi )^n}+
  \cos\left(\df{\pi n}{2}\right) \sum_{j=1}^{\infty}
\df{p^2+\lambda_j^2}{p(p+\frac{1}{\pi})+\lambda_j^2}
\df{\lambda_j^{n-1}}{\sigma\left(\frac{\alpha\lambda_j}{2\pi}\right)
e^{\alpha\lambda_j}-1}\right)\notag \\
    &=\sqrt{\b^n}\Bigg( \cos \left(\frac{\pi n}{2}\right)\frac{\Gamma(n)\eta_p(n)}{(2\pi)^n}+\sum_{j=1}^{\infty}\frac{p^2 + \l_{j}^2}{p(p+\frac{1}{\pi})+\l^2_{j}}\frac{\l_j^{n-1}}{\sigma \left(\frac{\b \l_j}{2\pi}\right)e^{\b \l_j}-1 }\nonumber \\
    &-i\sin \left(\frac{\pi n}{2}\right)\textup{P.V.}\int_{0}^{\infty}\frac{z^{n-1}}{\sigma(z)e^{2\pi z}-1} \left(\frac{\sigma \left(-\frac{i\b z}{2\pi}\right)e^{-i\b z}-\sigma \left(\frac{i\b z}{2\pi}\right)e^{i\b z} }{\left(\sigma \left(\frac{i\b z}{2\pi}\right)e^{i\b z}-1 \right)\left(\sigma \left(-\frac{i\b z}{2\pi}\right)e^{-i\b z}-1 \right)} \right)\ dz\Bigg).\label{bigidentity}
\end{align}
\end{theorem}

\begin{proof}
In the proof of Koshliakov's generalization of the Abel--Plana summation formula, we integrated over an indented rectangle with vertical sides passing through the real points $0$ and $r$, $r>0$, and horizontal sides passing through $\pm iN$, $N>0$.  In the present application we integrate over two (indented) rectangles.  One is in the upper half-plane with the upper side passing through $iN$, but with the lower side on the real axis.  The second is in the lower half-plane with its bottom side passing through $-iN$ and with the upper side on the real axis.
Let
\begin{equation}\label{b1}
\phi(z):=\df{z^{n-1}}{\sigma\left(\frac{\beta z}{2\pi}\right)e^{\beta z}-1},\qquad \Re(n)>2.
\end{equation}
The function $\phi(z)$ has simple poles at $z=\pm \frac{2\pi i \l_j}{\b}$, $1\leq j<\infty$, and a singularity at $z=0$.
We integrate
\begin{equation}
\df{\phi(z)}{\sigma(\mp iz)e^{\mp 2\pi iz}-1}
\end{equation}
over, respectively, the contours in the upper and lower half-planes, respectively.  In the first instance, we integrate in the positive, counter-clockwise direction, and in the second case we integrate in the clockwise direction.  The integrand has simple poles on the real axis at $z=\lambda_j$, $j\geq 1$, poles on the imaginary axis at $z=\pm \frac{2\pi i\lambda_j}{\beta}$, $j\geq 1$, and a singularity at the origin.  For each pole on the real axis, when we integrate in the upper half-plane, we take a semi-circular indentation of radius $\epsilon$, $0<\epsilon<1$, in the upper half-plane; when we integrate in the lower half-plane, the semi-circular indentation about the pole will be in the lower half-plane. For the poles on the imaginary axis, the contours have semi-circular indentations of radius $\epsilon$, $0<\epsilon<1$, in the right half-plane.  Lastly, for the singularity at the origin, when integrating in the upper half-plane, the contour has a quarter-circular indentation of radius $\epsilon$, $0<\epsilon<1$, about $0$ lying in the upper half-plane, and when integrating in the lower half-plane, the contour has a quarter-circular indentation of radius $\epsilon$, $0<\epsilon<1$, around $0$ in the lower half-plane.  Let $\mathbb{C}^{\pm}$ denote, respectively, the aforementioned contours in the upper and lower half-planes.  Since $\phi(z)$ has no other singularities inside the contours $\mathbb{C}^{\pm}$, by Cauchy's Theorem,
\begin{equation}\label{b2}
\int_{\mathbb{C}^{\pm}}\df{\phi(z)}{\sigma(\mp iz)e^{\mp 2\pi iz}-1}=0.
\end{equation}

 We first determine the three terms on the right-hand side of \eqref{koshman1}. Because $\Re(n)>2$, we see that $\phi(0)=0$.
From \cite[p.~21, Equation (31)]{koshliakov3}
\begin{align}\label{21}
    \int_{0}^{\infty}\frac{z^{n-1}}{\sigma(z)e^{2\pi z}-1} \ dz=\frac{\Gamma(n)\eta_p(n)}{(2\pi)^n},
    \end{align}
where $\eta_p(s)$ is defined in \eqref{eps}. Thus, by \eqref{21},
\begin{align}\label{21a}
    \int_{0}^{\infty}\phi(x)\ dx=  \int_{0}^{\infty}\frac{z^{n-1}}{\sigma\left(\frac{\b z}{2\pi } \right)e^{\b z}-1} \ dz=\frac{\Gamma(n)\eta_p(n)}{(\b)^n}.
\end{align}

Consider next the second integral on the right side of \eqref{koshman1}.
Using \eqref{sigma} below, we find that
\begin{align}
    \phi (i z)&-\phi(-i z) \nonumber\\
    =&z^{n-1}\left(\frac{e^{\pi i(n-1)/2}}{\sigma \left(\frac{i\b z}{2\pi}\right)e^{i\b z}-1} -\frac{e^{-\pi i(n-1)/2}}{\sigma \left(-\frac{i\b z}{2\pi}\right)e^{-i\b z}-1} \right)\nonumber\\
    =&\frac{z^{n-1}}{\left(\sigma \left(\frac{i\b z}{2\pi}\right)e^{i\b z}-1 \right)\left(\sigma \left(-\frac{i\b z}{2\pi}\right)e^{-i\b z}-1 \right)}\nonumber \\
    &\times\Bigg\{i \cos\left( \frac{\pi n}{2}\right)\left(\sigma \left(\frac{i\b z}{2\pi}\right)e^{i\b z}-1 \right)\left(\sigma \left(-\frac{i\b z}{2\pi}\right)e^{-i\b z}-1 \right)\nonumber \\
    &\qquad+\sin\left(\frac{\pi n}{2}\right)\left(\sigma \left(-\frac{i\b z}{2\pi}\right)e^{-i\b z}-\sigma \left(\frac{i\b z}{2\pi}\right)e^{i\b z} \right) \Bigg\}\nonumber\\
    =&z^{n-1}\left\{i \cos\left( \frac{\pi n}{2}\right)+\sin \left( \frac{\pi n}{2}\right)\left(\frac{\sigma \left(-\frac{i\b z}{2\pi}\right)e^{-i\b z}-\sigma \left(\frac{i\b z}{2\pi}\right)e^{i\b z} }{\left(\sigma \left(\frac{i\b z}{2\pi}\right)e^{i\b z}-1 \right)\left(\sigma \left(-\frac{i\b z}{2\pi}\right)e^{-i\b z}-1 \right)} \right) \right\}.\nonumber
\end{align}
From examinations of the integrands' singularities below, we observe that ordinary Riemann integrals do not exist.  However, the principal values of the integrals do exist.  Thus,
\begin{align}\label{proof1}
    i~ &\textup{P.V.}\int_{0}^{\infty}\frac{\phi (i z)-\phi(-i z)}{\sigma(z)e^{2\pi z}-1}\ dz =-\cos \left(\frac{\pi n}{2}\right)\int_{0}^{\infty}\frac{z^{n-1}}{\sigma(z)e^{2\pi z}-1} \ dz \nonumber \\
    &+i\sin \left(\frac{\pi n}{2}\right)\textup{P.V.}\int_{0}^{\infty}\frac{z^{n-1}}{\sigma(z)e^{2\pi z}-1} \left(\frac{\sigma \left(-\frac{i\b z}{2\pi}\right)e^{-i\b z}-\sigma \left(\frac{i\b z}{2\pi}\right)e^{i\b z} }{\left(\sigma \left(\frac{i\b z}{2\pi}\right)e^{i\b z}-1 \right)\left(\sigma \left(-\frac{i\b z}{2\pi}\right)e^{-i\b z}-1 \right)} \right)\ dz.
\end{align}
Applying \eqref{21}
in \eqref{proof1}, we obtain
\begin{align}
    i~ &\textup{P.V.}\int_{0}^{\infty}\frac{\phi (i z)-\phi(-i z)}{\sigma(z)e^{2\pi z}-1}\ dz =-\cos \left(\frac{\pi n}{2}\right)\frac{\Gamma(n)\eta_p(n)}{(2\pi)^n} \label{pv} \\
    &+i\sin \left(\frac{\pi n}{2}\right)\textup{P.V.}\int_{0}^{\infty}\frac{z^{n-1}}{\sigma(z)e^{2\pi z}-1} \left(\frac{\sigma \left(-\frac{i\b z}{2\pi}\right)e^{-i\b z}-\sigma \left(\frac{i\b z}{2\pi}\right)e^{i\b z} }{\left(\sigma \left(\frac{i\b z}{2\pi}\right)e^{i\b z}-1 \right)\left(\sigma \left(-\frac{i\b z}{2\pi}\right)e^{-i\b z}-1 \right)} \right)\ dz.\notag
\end{align}

We now calculate the contributions from \eqref{b2}.
Consider
\begin{equation}\label{b3}
\int_{\mathbb{C}^{\pm}}\dfrac{z^{n-1}dz}{(\sigma \left(\frac{\b z}{2\pi}\right)e^{\b z}-1 )(\sigma(\mp iz)e^{\mp 2\pi iz}-1)}, \qquad \Re(n)>2.
\end{equation}
The integrand has simple poles at $z=\lambda_j$, $1\leq j<\infty$.  For each $\lambda_j$,  we first calculate the limit of the integral over its related semi-circular contour in the upper half-plane as $\epsilon \to 0$. Set $z=\lambda_j+\epsilon e^{i\theta}, \pi\geq \theta \geq 0$.  Thus, we determine
\begin{align}\label{b4}
J_1:=&\lim_{\epsilon  \to 0}\int_{\pi}^0\df{(\lambda_j+\epsilon e^{i\theta})^{n-1}\epsilon ie^{i\theta}d\theta}
{\left(\sigma\left(\frac{\beta}{2\pi}(\lambda_j+\epsilon e^{i\theta})\right)
e^{\beta(\lambda_j+\epsilon e^{i\theta})}-1\right)
\left(\sigma(- i(\lambda_j+\epsilon e^{i\theta}))e^{- 2\pi i(\lambda_j+\epsilon e^{i\theta})}-1\right)}\notag\\
=&\df{i\,\lambda_j^{n-1}}{\left(\sigma\left(\frac{\beta\lambda_j}{2\pi}\right)e^{\beta\lambda_j}-1\right)}
\int_{\pi}^0e^{i\theta}d\theta\lim_{\epsilon  \to 0}\df{\epsilon}{\left(\sigma(- i(\lambda_j+\epsilon e^{i\theta}))e^{- 2\pi i(\lambda_j+\epsilon e^{i\theta})}-1\right)}.
\end{align}
To calculate the limit above, we invoke the partial fraction decomposition
  \cite[p.~14, Equation 4]{koshliakov3}
\begin{align}\label{2.4}
\frac{1}{\sigma(z)e^{2\pi z}-1}=-\frac{1}{2}+\frac{1}{2 \pi}\frac{1}{1+\frac{1}{\pi p}}\frac{1}{z}+\frac{z}{\pi}\sum_{k=1}^{\infty}\frac{p^2+\lambda^2_{k}}{p\left(p+\frac{1}{\pi}\right)
+\lambda^2_k}\frac{1}{z^2+\lambda^2_k}.
\end{align}
Thus, from \eqref{2.4},
\begin{gather}
\df{1}{\left(\sigma(- i(\lambda_j+\epsilon e^{i\theta}))e^{- 2\pi i(\lambda_j+\epsilon e^{i\theta})}-1\right)}
=-\df12+\df{1}{2\pi}\,\df{1}{1+\tf{1}{\pi p}}\,
\df{1}{-i(\lambda_j+\epsilon e^{i\theta})}\notag\\
+\df{-i(\lambda_j+\epsilon e^{i\theta})}{\pi}
\sum_{k=1}^{\infty}\frac{p^2+\lambda^2_{k}}{p\left(p+\frac{1}{\pi}\right)
+\lambda^2_k}\,\frac{1}{(-i(\lambda_j+\epsilon e^{i\theta}))^2+\lambda^2_k}\label{2.4cc}.
\end{gather}
When we multiply both sides of \eqref{2.4cc} by $\epsilon$ and take the limit as $\epsilon\to0$, the only limiting expression that is not equal to $0$ arises from the term when $k=j$ in the series.  Thus,
\begin{align}\label{2.4ccc}
&\lim_{\epsilon\to 0}\df{\epsilon}{\left(\sigma(- i(\lambda_j+\epsilon e^{i\theta}))e^{- 2\pi i(\lambda_j+\epsilon e^{i\theta})}-1\right)}\notag\\=&-\lim_{\epsilon\to 0}\df{\epsilon i\lambda_j}{\pi}\,\frac{p^2+\lambda^2_{j}}{p\left(p+\frac{1}{\pi}\right)
+\lambda^2_j}\,\frac{1}{(-i(\lambda_j+\epsilon e^{i\theta}))^2+\lambda^2_j}\notag\\
=&-\lim_{\epsilon\to 0}\df{\epsilon i \lambda_j}{\pi}\,\frac{p^2+\lambda^2_{j}}{p\left(p+\frac{1}{\pi}\right)
+\lambda^2_j}\,\frac{1}{-2\epsilon\lambda_j e^{i\theta}-\epsilon^2 e^{2i\theta}}\notag\\
=&\df{i(p^2+\lambda_j^2)}{2\pi e^{i\theta} \left(p\left(p+\frac{1}{\pi}\right)+\lambda_j^2\right)}.
\end{align}
Hence, by \eqref{b4} and \eqref{2.4ccc},
\begin{align}\label{2.4cccc}
J_1=&
-\df{\lambda_j^{n-1}}{\left(\sigma\left(\frac{\beta\lambda_j}{2\pi}\right)e^{\beta\lambda_j}-1\right)}
\int_{\pi}^0\df{p^2+\lambda_j^2}{2\pi \left(p\left(p+\frac{1}{\pi}\right)+\lambda_j^2\right)}d\theta\notag\\
=&\df{\lambda_j^{n-1}}{2\left(\sigma\left(\frac{\beta\lambda_j}{2\pi}\right)e^{\beta\lambda_j}-1\right)}\,
\df{p^2+\lambda_j^2}{\left(p\left(p+\frac{1}{\pi}\right)+\lambda_j^2\right)}.
\end{align}
Upon letting $r\to\infty$, we find that the total number of contributions from the semi-circular contours in the upper half-plane is equal to
\begin{equation}\label{2.4dd}
\mathbb{J}_1:=\df12\sum_{j=1}^{\infty}
\df{\lambda_j^{n-1}}{\left(\sigma\left(\frac{\beta\lambda_j}{2\pi}\right)e^{\beta\lambda_j}-1\right)}\,
\df{p^2+\lambda_j^2}{\left(p\left(p+\frac{1}{\pi}\right)+\lambda_j^2\right)}.
\end{equation}

 Secondly, for each $\lambda_j$, $1\leq j<\infty$, we calculate the limit as  $\epsilon\to0$ of the integral over the semi-circular contours in the lower half-plane. Set $z=\lambda_j+\epsilon e^{i\theta}, \pi\leq \theta \leq 2\pi$.  We determine
\begin{align}\label{b4b}
J_2:=&\lim_{\epsilon  \to 0}\int_{\pi}^{2\pi}\df{(\lambda_j+\epsilon e^{i\theta})^{n-1}\epsilon ie^{i\theta}d\theta}
{\left(\sigma\left(\frac{\beta}{2\pi}(\lambda_j+\epsilon e^{i\theta})\right)
e^{\beta(\lambda_j+\epsilon e^{i\theta})}-1\right)
\left(\sigma( i(\lambda_j+\epsilon e^{i\theta}))e^{ 2\pi i(\lambda_j+\epsilon e^{i\theta})}-1\right)}\notag\\
=&\lim_{\epsilon  \to 0}\df{\epsilon i\,\lambda_j^{n-1}}{\left(\sigma\left(\frac{\beta\lambda_j}{2\pi}\right)e^{\beta\lambda_j}-1\right)}
\int_{2\pi}^{\pi}\df{e^{i\theta}d\theta}{\left(\sigma( i(\lambda_j+\epsilon e^{i\theta}))e^{ 2\pi i(\lambda_j+\epsilon e^{i\theta})}-1\right)}.
\end{align}
The remaining portions of the calculation are analogous to those for $J_1$, and we find that
\begin{equation*}
J_2=\df{\lambda_j^{n-1}}{2\left(\sigma\left(\frac{\beta\lambda_j}{2\pi}\right)e^{\beta\lambda_j}-1\right)}\,
\df{p^2+\lambda_j^2}{\left(p\left(p+\frac{1}{\pi}\right)+\lambda_j^2\right)}.
\end{equation*}
As above, upon letting $r\to\infty$, we find that the total number of contributions from the semi-circular contours in the lower half-plane is equal to
\begin{equation}\label{2.4dddd}
\mathbb{J}_2:=\df12\sum_{j=1}^{\infty}
\df{\lambda_j^{n-1}}{\left(\sigma\left(\frac{\beta\lambda_j}{2\pi}\right)e^{\beta\lambda_j}-1\right)}\,
\df{p^2+\lambda_j^2}{\left(p\left(p+\frac{1}{\pi}\right)+\lambda_j^2\right)}.
\end{equation}
Therefore, from \eqref{b4b} and \eqref{2.4dd}, the contribution from the poles on the real axis is
\begin{equation}\label{2.4ddd}
\mathbb{J}_1+\mathbb{J}_2=\sum_{j=1}^{\infty}
\df{\lambda_j^{n-1}}{\left(\sigma\left(\frac{\beta\lambda_j}{2\pi}\right)e^{\beta\lambda_j}-1\right)}\,
\df{p^2+\lambda_j^2}{\left(p\left(p+\frac{1}{\pi}\right)+\lambda_j^2\right)}.
\end{equation}

We now calculate the contributions from the poles on the imaginary axis. We first consider the indentation around the pole at $\frac{2\pi i \l_j}{\b}, ~1\leq j<\infty$.
Letting $z=\frac{2\pi i \l_j}{\b}+\epsilon e^{i \theta} $, we define
\begin{align}\label{L10}
    L_1:&=\lim_{\epsilon \to 0}\int_{\pi/2}^{-\pi/2 }\tfrac{\left( \frac{2\pi i \l_j}{\b}+\epsilon e^{i \theta} \right)^{n-1}i \epsilon e^{i \theta}\ d\theta}{\left(\sigma \left(\frac{\b }{2\pi}\left( \frac{2\pi i \l_j}{\b}+\epsilon e^{i \theta} \right)\right)e^{\b \left( \frac{2\pi i \l_j}{\b}+\epsilon e^{i \theta} \right)}-1 \right)\left(\sigma \left(-i \left( \frac{2\pi i \l_j}{\b}+\epsilon e^{i \theta} \right)\right)e^{-2\pi i \left( \frac{2\pi i \l_j}{\b}+\epsilon e^{i \theta} \right)}-1 \right)} \nonumber \\
    &=\int_{\pi/2}^{-\pi/2 }\tfrac{\left( \frac{2\pi i \l_j}{\b}\right)^{n-1}i e^{i \theta}}{\sigma\left( \frac{2\pi \l_j}{\b}\right)e^{\frac{4\pi^2 \l_j}{\b}}-1}\mathbb{L}_1(\theta)\ d\theta,
\end{align}
where
\begin{align}\label{Ltheta}
    \mathbb{L}_1(\theta):&=\lim_{\epsilon \to 0}\frac{\epsilon}{\left(\sigma \left(\frac{\b }{2\pi}\left( \frac{2\pi i \l_j}{\b}+\epsilon e^{i \theta} \right)\right)e^{\b \left( \frac{2\pi i \l_j}{\b}+\epsilon e^{i \theta} \right)}-1 \right)}\nonumber \\
    &=\lim_{\epsilon \to 0}\frac{\epsilon}{\sigma \left( i \l_j+\frac{\b \epsilon e^{i \theta}}{2\pi } \right) e^{2\pi i \l_j+\b \epsilon e^{i \theta}}-1 }.
\end{align}
Appealing again to \eqref{2.4}, we find that
\begin{align}
\label{2.4a}
\mathbb{L}_1(\theta):=&\lim_{\epsilon\to 0}\epsilon\bigg(-\df12+\df{1}{2\pi}\df{1}{\left(1+\frac{1}{p}\right)}\,\df{1}{\left(i\lambda_j+\frac{\beta\epsilon e^{i\theta}}{2\pi}\right)}\notag\\
&+\df{i\lambda_j+\frac{\beta\epsilon e^{i\theta}}{2\pi}}{\pi}
\sum_{k=1}^{\infty}\df{p^2+\lambda_k^2}{p(p+\frac{\pi}{p})+\lambda_k^2}\,
\df{1}{\left(i\lambda_j+\frac{\beta\epsilon e^{i\theta}}{2\pi}\right)^2+\lambda_k^2}\bigg).
\end{align}
As $\epsilon\to0$, the only  term from the partial fraction decomposition \eqref{2.4a} that provides a nonzero contribution in the limit  is when $k=j$. Thus,
\begin{align}\label{2.4b}
\mathbb{L}_1(\theta)=&\lim_{\epsilon\to 0}\epsilon
\df{i\lambda_j+\frac{\beta\epsilon e^{i\theta}}{2\pi}}{\pi}\,
\df{p^2+\lambda_j^2}{p(p+\frac{1}{\pi})+\lambda_j^2}\,
\df{1}{\left(i\lambda_j+\frac{\beta\epsilon e^{i\theta}}{2\pi}\right)^2+\lambda_j^2}\notag\\
=&\lim_{\epsilon\to 0}\epsilon
\df{i\lambda_j+\frac{\beta\epsilon e^{i\theta}}{2\pi}}{\pi}\,
\df{p^2+\lambda_j^2}{p(p+\frac{1}{\pi})+\lambda_j^2}\,
\df{1}{\df{i\lambda_j\beta\epsilon e^{i\theta}}{\pi}+\left(\df{\beta\epsilon e^{i\theta}}{2\pi}\right)^2}\notag\\
=&\df{i\lambda_j}{\pi}\df{p^2+\lambda_j^2}{p(p+\frac{1}{\pi})+\lambda_j^2}\,
\df{1}{\df{i\lambda_j\beta e^{i\theta}}{\pi}}\notag\\
=&\df{p^2+\lambda_j^2}{\beta e^{i\theta}\left(p\left(p+\frac{1}{\pi}\right)+\lambda_j^2\right)}.
\end{align}
Putting \eqref{2.4b} in \eqref{L10} and using the hypothesis $\alpha\beta=4\pi^2$, we conclude that
\begin{align}\label{2.4c}
L_1=&\int_{\pi/2}^{-\pi/2}\df{\left(\frac{2\pi i\lambda_j}{\beta}\right)^{n-1}ie^{i\theta}}
{\sigma\left(\frac{2\pi\lambda_j}{\beta}\right)e^{\frac{4\pi^2\lambda_j}{\beta}}-1}\,\,
\df{p^2+\lambda_j^2}{\beta e^{i\theta}\left(p\left(p+\frac{1}{\pi}\right)+\lambda_j^2\right)}\,d\theta\notag\\
=&-\df{i\pi(p^2+\lambda_j^2)\left(\frac{2\pi i\lambda_j}{\beta}\right)^{n-1}}
{\beta\left(p\left(p+\frac{1}{\pi}\right)+\lambda_j^2\right)
\left(\sigma\left(\frac{2\pi\lambda_j}{\beta}\right)e^{\alpha\lambda_j}-1\right)}.
\end{align}
Now,
\begin{equation*}
-\df{i\pi}{\beta}\left(\frac{2\pi i\lambda_j}{\beta}\right)^{n-1}
=-\df{(2\pi)^n\lambda_j^{n-1}i^n}{2\beta^n}=-\df12\left(\df{4\pi^2}{\beta^2}\right)^{n/2}\lambda_j^{n-1}i^n
=-\df12\left(\df{\alpha}{\beta}\right)^{n/2}\lambda_j^{n-1}i^n.
\end{equation*}
Thus, \eqref{2.4c} takes the shape
\begin{equation}\label{2.4d}
L_1=-\df12\left(\df{\alpha}{\beta}\right)^{n/2}\lambda_j^{n-1}i^n\df{p^2+\lambda_j^2}
{\left(p\left(p+\frac{1}{\pi}\right)+\lambda_j^2\right)
\left(\sigma\left(\frac{2\pi\lambda_j}{\beta}\right)e^{\alpha\lambda_j}-1\right)}.
\end{equation}

 Secondly, consider the indentation around the pole at $\frac{-2\pi i \l_j}{\b}, ~1\leq j<\infty$.
Let $z=-\frac{2\pi i \l_j}{\b}+\epsilon e^{i \theta} $ and define
\begin{align}\label{2.4f}
    L_2:&=\lim_{\epsilon \to 0}\int_{-\pi/2}^{\pi/2 }\tfrac{\left( -\frac{2\pi i \l_j}{\b}+\epsilon e^{i \theta} \right)^{n-1}i \epsilon e^{i \theta}}{\left(\sigma \left(\frac{\b }{2\pi}\left( -\frac{2\pi i \l_j}{\b}+\epsilon e^{i \theta} \right)\right)e^{\b \left(- \frac{2\pi i \l_j}{\b}+\epsilon e^{i \theta} \right)}-1 \right)\left(\sigma \left(i \left( -\frac{2\pi i \l_j}{\b}+\epsilon e^{i \theta} \right)\right)e^{2\pi i \left(- \frac{2\pi i \l_j}{\b}+\epsilon e^{i \theta} \right)}-1 \right)} \nonumber \\
    &=:\int_{-\pi/2}^{\pi/2 }\tfrac{\left( -\frac{2\pi i \l_j}{\b}\right)^{n-1}i e^{i \theta}}{\sigma\left( \frac{2\pi \l_j}{\b}\right)e^{4\pi^2 \l_j/\b}-1}\mathbb{L}_2(\theta)\ d\theta.
\end{align}
To calculate $\mathbb{L}_2(\theta)$, we use \eqref{2.4} and proceed as in \eqref{2.4a} to find that
\begin{align}\label{2.4bc}
\mathbb{L}_2(\theta)=&\lim_{\epsilon\to 0}\epsilon
\df{-i\lambda_j+\frac{\beta\epsilon e^{i\theta}}{2\pi}}{\pi}\,
\df{p^2+\lambda_j^2}{p(p+\frac{1}{\pi})+\lambda_j^2}\,
\df{1}{\left(-i\lambda_j+\frac{\beta\epsilon e^{i\theta}}{2\pi}\right)^2+\lambda_j^2}\notag\\
=&\lim_{\epsilon\to 0}\epsilon
\df{-i\lambda_j+\frac{\beta\epsilon e^{i\theta}}{2\pi}}{\pi}\,
\df{p^2+\lambda_j^2}{p(p+\frac{1}{\pi})+\lambda_j^2}\,
\df{1}{\left(-\frac{i\lambda_j\beta\epsilon e^{i\theta}}{\pi}+\left(\frac{\beta\epsilon e^{i\theta}}{2\pi}\right)^2\right)}\notag\\
=&-\df{i\lambda_j}{\pi}\,\df{p^2+\lambda_j^2}{p(p+\frac{1}{\pi})+\lambda_j^2}\,
\frac{1}{\frac{-i\lambda_j\beta\e^{i\theta}}{\pi}}\notag\\
=&\frac{p^2+\lambda_j^2}{\beta e^{i\theta}\left(p\left(p+\frac{1}{\pi}\right)+\lambda_j^2\right)}.
\end{align}
Using \eqref{2.4bc} in \eqref{2.4f}, and once again using the hypothesis $\alpha\beta=4\pi^2$, we conclude that
\begin{align}\label{2.4e}
L_2=&\pi i\df{\left(-\frac{2\pi i\lambda_j}{\beta}\right)^{n-1}}{\beta}\,
\df{p^2+\lambda_j^2}{p(p+\frac{1}{\pi})+\lambda_j^2}\,
\df{1}{\sigma\left(\frac{2\pi\lambda_j}{\beta}\right)e^{\alpha\lambda_j}-1}\notag\\
=&-\df12\left(\df{\alpha}{\beta}\right)^{n/2}\lambda_j^{n-1}(-i)^n
\df{p^2+\lambda_j^2}{p(p+\frac{1}{\pi})+\lambda_j^2}\,
\df{1}{\sigma\left(\frac{2\pi\lambda_j}{\beta}\right)e^{\alpha\lambda_j}-1}.
\end{align}

Hence, by \eqref{2.4d} and \eqref{2.4e},
\begin{align}\label{2.4g}
L_1+L_2=-&\df12\left(\df{\alpha}{\beta}\right)^{n/2}\lambda_j^{n-1}\left((-i)^n+(i)^n\right)
\df{p^2+\lambda_j^2}{p(p+\frac{1}{\pi})+\lambda_j^2}\,
\df{1}{\sigma\left(\frac{2\pi\lambda_j}{\beta}\right)e^{\alpha\lambda_j}-1}\notag\\
=&-\left(\df{\alpha}{\beta}\right)^{n/2}\lambda_j^{n-1}\cos(\tfrac12\pi n)
\df{p^2+\lambda_j^2}{p(p+\frac{1}{\pi})+\lambda_j^2}\,
\df{1}{\sigma\left(\frac{2\pi\lambda_j}{\beta}\right)e^{\alpha\lambda_j}-1}.
\end{align}
Also, recall that $\phi(0)=0$. Hence, there is no contribution from the singularity at the origin.

Finally, from \eqref{2.4g}, \eqref{pv}, and Theorem \ref{theorem4}, we conclude that

\begin{align}
    &\sum_{j=1}^{\infty}\frac{p^2 + \l_{j}^2}{p(p+\frac{1}{\pi})+\l^2_{j}}\frac{\l_j^{n-1}}{\sigma \left(\frac{\b \l_j}{2\pi}\right)e^{\b \l_j}-1 }\notag\\
   &- \left(\df{\alpha}{\beta}\right)^{n/2}\cos\left(\df{\pi n}{2}\right)\sum_{j=1}^{\infty}
\df{p^2+\lambda_j^2}{p(p+\frac{1}{\pi})+\lambda_j^2}
\df{\lambda_j^{n-1}}{\sigma\left(\frac{2\pi\lambda_j}{\beta}\right)
e^{\alpha\lambda_j}-1}\notag\\
    =& \frac{\Gamma(n)\eta_p(n)}{(\b)^n}-\cos \left(\frac{\pi n}{2}\right)\frac{\Gamma(n)\eta_p(n)}{(2\pi)^n} \nonumber \\
    &+i\sin \left(\frac{\pi n}{2}\right)\textup{P.V.}\int_{0}^{\infty}\frac{z^{n-1}}{\sigma(z)e^{2\pi z}-1} \left(\frac{\sigma \left(-\frac{i\b z}{2\pi}\right)e^{-i\b z}-\sigma \left(\frac{i\b z}{2\pi}\right)e^{i\b z} }{\left(\sigma \left(\frac{i\b z}{2\pi}\right)e^{i\b z}-1 \right)\left(\sigma \left(-\frac{i\b z}{2\pi}\right)e^{-i\b z}-1 \right)} \right) dz,\label{bigthm}
\end{align}
which is readily seen to be equivalent to Theorem \ref{thm1}.
\end{proof}

We now show that Entry \ref{ab} is a special case of Theorem \ref{thm1}.  Let $p\to\infty$.  Then, $\lambda_j\to j$, $\sigma(z)\to 1$, and $\eta_p(s)\to\zeta(s)$ \cite[p.~6]{kzf}, where $\zeta(s)$ denotes the Riemann zeta function. Also, a brief calculation gives
$$\df{e^{-i\beta z}-e^{i\beta z}}{(e^{i\beta z}-1)(e^{-i\beta z}-1)}=-i\cot(\tf12 \beta z).$$
Letting $p\to\infty$ and using the observations above in Theorem \ref{thm1}, we deduce Entry \ref{ab}.

Next, let $p\to 0$. Then, $\lambda_j\to j-\frac12$, $\sigma(z)\to -1$, and $\eta_p(s) \to (2^{1-s}-1)\zeta(s)$ \cite[p.~6, Equation (2.11)]{kzf}. Another brief calculation gives
$$\df{-e^{-i\beta z}+e^{i\beta z}}{(-e^{i\beta z}-1)(-e^{-i\beta z}-1)}=i\tan(\tf12 \beta z).$$
Hence, Equation \eqref{bigthm} takes the shape
\begin{gather*}
-\sum_{j=1}^{\infty}\frac{(j-\tf12)^{n-1}}{e^{\b\left(j-\tf12\right)}+1 }= \frac{\Gamma(n)(2^{1-n}-1)\zeta(n)}{(\b)^n}-\cos \left(\frac{\pi n}{2}\right)\frac{\Gamma(n)(2^{1-n}-1)\zeta(n)}{(2\pi)^n}  \\
    +\sin \left(\frac{\pi n}{2}\right)\textup{P.V.}\int_{0}^{\infty}\frac{z^{n-1}}{e^{2\pi z}+1}\tan(\tf12\beta z) \ dz
   -\left(\df{\alpha}{\beta}\right)^{n/2}\cos\left(\df{\pi n}{2}\right)\sum_{j=1}^{\infty}\df{(j-\tf12)^{n-1}}
{e^{\alpha(j-\tf12)}+1}.
\end{gather*}

Rewriting the aforementioned identity, we record it in the following corollary, which is new.

\begin{corollary} Let Re$(n)>2$, $\alpha,\beta>0$ such that $\alpha\beta=4\pi^2$.  Then
\begin{gather*}
\sum_{j=1}^{\infty}\frac{(j-\tf12)^{n-1}}{e^{\b (j-\tf12)}+1 }+
\frac{\Gamma(n)(2^{1-n}-1)\zeta(n)}{(\b)^n}\\
=\cos\left(\df{\pi n}{2}\right)\left(\left(\df{\alpha}{\beta}\right)^{n/2}
\sum_{j=1}^{\infty}\df{(j-\tf12)^{n-1}}{e^{\alpha(j-\tf12)}+1}
+\frac{\Gamma(n)(2^{1-n}-1)\zeta(n)}{(2\pi)^n}\right)\\
-\sin \left(\frac{\pi n}{2}\right)\textup{P.V.}\int_{0}^{\infty}\frac{z^{n-1}}{e^{2\pi z}+1}\tan(\tf12\beta z) \ dz.
\end{gather*}
\end{corollary}
We conclude the paper with one more corollary of Theorem \ref{thm1}, which was obtained in \cite[5.1]{kzf}.
\begin{corollary}\label{1cor}
For $m\in\mathbb{N}$ and $\a\b=\pi^2$,
\begin{align}\label{1coreqn}
\a^{m+1}&\left\{\frac{1}{2}\zeta_{p}(-2m-1)+
\sum_{j=1}^{\infty}\frac{p^2+\l_j^2}{p\left(p+\frac{1}{\pi}\right)
+\l_j^2}\cdot\frac{\l_{j}^{2m+1}}{\s\left(\frac{\l_j \a}{\pi} \right)e^{2\a \l_j}-1}\right\}\nonumber\\
&=(-\b)^{m+1}\left\{\frac{1}{2}\zeta_{p}(-2m-1)
+\sum_{j=1}^{\infty}\frac{p^2+\l_j^2}{p\left(p+\frac{1}{\pi}\right)+\l_j^2}
\cdot\frac{\l_{j}^{2m+1}}{\s\left(\frac{\l_j \b}{\pi} \right)e^{2\b \l_j}-1}\right\}.
\end{align}
\end{corollary}
\begin{proof}
Recall the definitions of $\zeta_p(s)$ and $\eta_p(s)$ in \eqref{zps} and \eqref{eps}, respectively. Koshliakov \cite[p.~20, Chapter 1, Equation (30)]{koshliakov3} showed that these generalized zeta functions are related to each other by means of the  functional equation,
\begin{align}\label{2.30}
\zeta_{p}(1-s)=\frac{2 \cos\left(\frac{\pi s}{2}\right)\Gamma(s)}{(2\pi)^s}\eta_{p}(s).
\end{align}
Now let $n$ be an even integer greater than $2$, say, $n=2m, m>1,$ in Theorem \ref{thm1}. Then employ \eqref{2.30} to write $\eta_p(2m)$ in terms of $\zeta_p(1-2m)$ and finally replace $m$ by $m+1$ to arrive at \eqref{1coreqn}.
\end{proof}
We note that the special case $p\to0$ of Corollary \ref{1cor} was first obtained by Malurkar \cite{malurkar} and later rediscovered by the first author \cite[Theorem 4.7]{berndtcrelle}.
\vspace{.1in}

\begin{center}
\textbf{Acknowledgements} \end{center}

The authors are very grateful to the University of Illinois Mathematics Department Librarians, Bernadette Braun, Becky Burner, and Tim Cole for securing some of the cited  manuscripts of N.~S.~Koshliakov as well as the reference \cite{saaryan1}. 

The first and the second authors' research is supported by the MHRD SPARC grant SPARC/2018-2019/P567/SL. The first author's research is partially supported by the Simons Foundation, whereas that of the second author is partially supported by SERB MATRICS grant MTR/2018/000251 and CRG grant CRG/2020/002367.

\end{document}